\documentclass{article}

\usepackage[latin1]{inputenc}
\usepackage[english]{babel}

\usepackage{amsmath}
\usepackage{amsfonts}
\usepackage{amssymb}
\usepackage{amscd}
\usepackage{xypic}
\usepackage[all]{xy}
\usepackage{mathrsfs}
\usepackage{latexsym}

\usepackage[mathscr]{eucal}

\usepackage{amsthm}

\usepackage{setspace}                    

\usepackage{hyperref}

\xyoption{all} 

\numberwithin{equation}{section}

\newtheorem{theorem}{Theorem}[section]
\newtheorem{corollary}[theorem]{Corollary}
\newtheorem{proposition}[theorem]{Proposition}
\newtheorem{lemma}[theorem]{Lemma}

\newtheorem{conjecture}[theorem]{Conjecture}

\theoremstyle{definition}

\newtheorem{definition}[theorem]{Definition}

\newtheorem{example}[theorem]{Example}

\newtheorem{remark}[theorem]{Remark}
\newtheorem{question}[theorem]{Question}

    \theoremstyle{plain}

    \newtheoremstyle{TheoremNum}
        {\topsep}{\topsep}              
        {\itshape}                      
        {}                              
        {\bfseries}                     
        {.}                             
        { }                             
        {\thmname{#1}\thmnote{ \bfseries #3}}
    \theoremstyle{TheoremNum}
    \newtheorem{thmn}{Theorem}
    \newtheorem{propn}{Proposition}
    \newtheorem{corn}{Corollary}

\newcommand{\nc}{\newcommand}
\nc{\fr}{{\rightarrow}}
\nc{\rd}{red.deg}

\newcommand{\pr}{\mathbb P}

\newcommand{\rk}{\mbox{\upshape{rank}}}

\newcommand{\picard}{\mathrm{Pic}}
\newcommand{\cliff}{\mathrm{Cliff}}
\newcommand{\oo}{{\cal O}}
\newcommand{\cE}{{\cal E}}
\newcommand{\cF}{{\cal F}}

\def\longepi{\relbar\joinrel\twoheadrightarrow}

\newif\ifpdf
\ifx\pdfoutput\undefined
\pdffalse
\else
\ifnum\pdfoutput=0
\pdffalse
\else
\pdfoutput=1 \pdftrue
\fi
\fi

         \ifpdf
         \usepackage[pdftex]{graphicx}
         \else
         \usepackage{graphicx}
         \fi

\newcommand{\grafica}{
        \ifpdf
        \DeclareGraphicsExtensions{.pdf, .jpg, .png}
        \else
        \DeclareGraphicsExtensions{.eps, .jpg}
        \fi}

\title{Linear series on curves: stability and Clifford index}
\author{E. C. Mistretta, L. Stoppino}

\begin{document}

\grafica

\maketitle

\begin{abstract}
We study  concepts of stabilities associated to a smooth complex curve together with a linear series on it.
In particular we investigate the relation between stability of the associated Dual Span Bundle and linear stability.
Our result implies a stability condition related to the Clifford index of the curve. 
Furthermore, in some of the cases, we prove that a stronger stability holds: cohomological stability.
Eventually using our results we obtain stable vector bundles of integral slope $3$, and  prove that they admit theta-divisors.
\end{abstract}

\section{Introduction}

Let $C$ be an irreducible projective smooth complex curve, and let ${\cal L}$ be a globally generated line bundle on $C$. 
Consider a generating subspace $V\subseteq H^0({\cal L})$.
The \emph{Dual Span Bundle} (DSB for short) $M_{V,{\cal L}}$ associated to this data is the kernel of the evaluation morphism:
$$
0\longrightarrow M_{V,{\cal L}}\longrightarrow V\otimes \oo_C\longrightarrow {\cal L}\longrightarrow 0.
$$
This is a vector bundle of rank $\dim V-1$ and degree $-\deg {\cal L}$. 
When $V = H^0({\cal L})$ we denote it $M_{\cal L}$.
This bundle is also called {\em transform} \cite{stabtrans}, or {\em Lazarsfeld bundle}  \cite{popa}.

In this paper we treat various kind of stability conditions, associated to these data, namely:
\begin{itemize}
\item[-] vector bundle stability, which we simply call stability, or slope stability, of $M_{V,{\cal L}}$ (Definition \ref{mustabilita});
\item[-] linear stability of the triple $(C,V,{\cal L})$ (Definition \ref{ls});
\item[-] cohomological stability of $M_{V,{\cal L}}$ (Definition \ref{cohostab}).
\end{itemize}

Stability of DSB's has been studied intensively and with many different purposes,
and it has been conjectured by Butler that it should  hold under generality conditions.
This conjecture has been verified in many cases,
and used to prove results on Brill-Noether theory and moduli spaces of coherent systems
(cf. \cite{Butler}, \cite{BP}, and \cite{BPN}).

These conditions satisfy the following implications
\begin{equation}\label{implicazioni}
cohomological \,\, stability \,\, \Rightarrow\,\, vector \,\, bundle \,\, stability \,\,\Rightarrow linear \,\, stability,
\end{equation}
which hold for semistability as well; moreover, as we are in characteristic  0m we have that cohomological semistability is equivalent to vector bundle semistability  \cite{EL}.

The purpose of this paper is twofold. Firstly we are interested in finding conditions under which the last implication in (\ref{implicazioni}) can be reversed, i.e. linear stability is a sufficient condition for the stability of the DSB.
The question of DSB's stability is considered by Butler in \cite{Butler}, and that work is the starting point of our investigation. It turns out that the Clifford index of the curve (definition in Section \ref{linear stability}) plays a central role.
In the last part of the paper we prove some counterexamples proving that the implication does not hold in general, and we state some conjectures. 

Secondly, we want to prove some new stability results. We establish some results, involving again the Clifford index. These results are achieved both using the arguments of the first part of the paper, and by different arguments for linear stability and cohomological stability.

\bigskip

Let us go deeper in the description of  our results.



As for the first question,
the convenience of reducing the stability of a DSB to the linear stability lies in the fact that linear stability is often less hard  to prove. 
Moreover, it has a clear geometric meaning in terms of relative degrees of projections of the given  morphism. 
So,  the question can be reformulated this way: to what extent the knowledge of geometry of a morphism is sufficient to detect the stability of the associated DSB?

Another  motivation for considering this problem comes from the work of the second author on fibred surfaces. To a fibred surface with a family of morphisms on the fibres, one can associate a certain divisorial class on the base curve. There are two methods that prove the positivity of this class, one assuming linear stability, the other assuming the stability of the DSB on a general fibre. The comparison between these methods leads naturally to comparing the two assumptions.


\medskip

Let us now assume that $C$ has genus $g\geq 2$. We obtain the following results. 

\begin{theorem}\label{main}
Let ${\cal L} \in \picard(C)$ be a globally generated line bundle,
and $V \subseteq H^0(C, \cal L)$ a generating space of global sections
such that 
\begin{equation}
\label{ahg}
\deg {{\cal L}} - 2 (\dim V  -1) \leqslant \textrm{Cliff} (C).
\end{equation}
Then linear (semi)stability of $(C, {\cal L}, V)$
is equivalent to (semi)stability 
of $M_{V,{\cal L}}$ in the following cases:

\begin{enumerate}

\item
\label{starthm} 
$ V = H^0({\cal L}) $ (complete case);

\item $\deg {\cal L}  \leqslant 2g - \textrm{Cliff} (C) +1$;

\item $\mathrm{codim}_{H^0({\cal L})} V <h^1({\cal L})+g/(\dim V-2)$;
 
 \item 
 \label{endthm}
 $\deg {\cal L} \geqslant 2g$, and $ \mathrm{codim}_{H^0({\cal L})} V \leqslant (\deg {\cal L} -2g)/2 $.

\end{enumerate}

\end{theorem}


The theorem is proved by applying a Castelnuovo type result, relating evaluation of sections of a line bundle $\mathcal{A}$ tensorized with the canonical bundle and the image of the morphism induced by global sections of $\mathcal{A}$, to an exact sequence obtained from a possible destabilization of the bundle  $M_{V,{\cal L}}$.
The geometrical idea for this construction is simple and is carried out in section \ref{cliff},
but the computations in order to make this argument work are quite long
(sections \ref{complete} and \ref{non-complete}), and give rise to the bounds imposed in points \ref{starthm} to \ref{endthm} of the main theorem.

Let us now describe the stability results that we obtain.

By standard linear series methods, we can prove in Section \ref{linear stability} the following dependance of linear stability on the Clifford index of the curve.

\begin{propn}[\ref{linear}]
Let $C$ be a curve of genus $g\geq 2$.
Let ${\cal L} \in \picard(C)$ be a globally generated line bundle
such that $\deg {\cal L} - 2 (h^0({\cal L}) -1) \leqslant \textrm{Cliff} (C)$.
Then ${\cal L}$ is linearly semistable. 
It is linearly stable unless 
${\cal L} \cong \omega_C(D)$ with $D$ an effective divisor of degree 2,
or
$C$ is hyperelliptic and  $\deg {\cal L} = 2 (h^0({\cal L}) -1)$.
\end{propn}

Using this result, and Theorem \ref{main},
we obtain stability of DSB in the following cases.

\begin{thmn}[\ref{paranjape}]
Let ${\cal L} \in \picard(C)$ be a globally generated line bundle
such that 
\begin{equation}
\deg {{\cal L}} - 2 (h^0({\cal L}) -1 )  \leqslant \textrm{Cliff} (C).
\end{equation}
Then $M_{\cal L}$ is semistable, and it is strictly semistable only in one of the following cases

(i)
${\cal L} \cong \omega_C (D)$ with $D$ an effective divisor of degree 2,

(ii)
$C$ is hyperelliptic and $\deg {{\cal L}}= 2 (h^0({\cal L}) -1)$.
\end{thmn}

In particular, this implies the following results.

\begin{corn}[\ref{spiego1}]
Let ${\cal L}$ be a  globally generated line bundle over $C$ with 
$$\deg {{\cal L}} \geqslant 2g - \cliff (C).$$
Then the vector bundle $M_{\cal L}$ is semistable. 
It is stable unless $(i)$ or $(ii)$ hold.
\end{corn}

\begin{corn}[\ref{spiego2}]
Let ${\cal L}$ be any line bundle that computes the Clifford index of $C$. Then $M_{\cal L}$ is semistable; it is stable unless $C$ is hyperelliptic.
\end{corn}

Moreover, using a result contained in \cite{barstop}, we can prove, applying Theorem \ref{main}, the following.

\begin{propn}[\ref{bs08}]
Let $C$ be a curve such that $\cliff (C) \geqslant 4$.
Let $V \subset H^0(\omega_C)$ be a general subspace of codimension 
smaller than or equal to 2,  then $M_{V, \omega_C}$ is semistable.
\end{propn}

Some of the results above were previously known
(\emph{e.g.} Theorem \ref{paranjape} is contained in Paranjape's PhD. thesis,
and Corollary \ref{spiego2} follows from \cite{BPOrt})

It is worthwhile remarking that in the paper \cite{stabtrans} the first author
proves stability of some bundles $M_{V, \cal L}$ by a similar argument: showing first that 
if $M_{V, \cal L}$ is unstable then $(C, {\cal L}, V)$ needs to be linearly unstable,
and then showing that for general $V \subset H^0(\cal L)$ these are not 
linearly unstable.

We hope these methods can be of use in order to verify the stability 
of dual span bundles in more cases,
and generalized to investigate on the stability of 
bundles which are dual span of higher rank vector bundles.

\medskip

Moreover, we prove in Section \ref{cohom} that,  in some of the cases of Theorem \ref{main}, a stronger condition holds:
\emph{cohomological stability} (Definition \ref{cohostab}).

\begin{thmn}[\ref{uffa}]
Let $({\cal L}, V)$ be a $g^r_d$ on a smooth curve $C$, inducing a birational morphism. Suppose that
\begin{itemize}
\item[-] $d\leqslant 2r+\cliff C$; 
\item[-] $\mathrm{codim}_{H^0({\cal L})} V \leqslant h^1({\cal L})$. 
\end{itemize}
Then $M_{V,{\cal L}}$ is cohomologically semistable. It is strictly stable unless $d=2r$. 
\end{thmn}
%
%
%


It is natural to wonder wether the implication linear stability $\Rightarrow $ stability of DSB holds more generally.
No examples -to our knowledge- were known where the first  stability condition holds while the second one does not.
The answer to this question is negative in general,
and it turned out to be fairly easy to produce linearly stable line bundles whose DSB is not semistable: this is the content of Section \ref{sezserie2} .

Eventually (Section \ref{theta}),  we show that on any curve $C$ of even genus $g=2k$,
having general gonality $\gamma (C) = k+1$ 
and general Clifford index $\textrm{Cliff} (C) = k-1$,
there exist stable DSB's of slope $-3$. If furthermore the curve $C$ is Petri,
we show that these admit a generalized theta-divisor.

%






%

\subsubsection*{Acknowledgements} 
This work was started as the second author was invited at Universit\`a di Padova on 
``CARIPARO - Progetti di Eccellenza'' fund.
We are deeply thankful to Leticia Brambila-Paz and Peter Newstead their advices and useful comments, the first author is very grateful for their kindness and hospitality during the semester MOS 2011 at Cambridge.

\subsubsection*{Notation} 

We will work over the complex numbers, 
and $C$ will be a smooth projective curve, unless explicitly specified.

Let  $D$ be a divisor on $C$. As customary, we shall write $H^i(D)$ for $H^i(\oo_C(D))$, and if ${\cal F} $ is a vector bundle  we shall use the notation  ${\cal F}(D)$ for ${\cal F}\otimes \oo_C(D)$.


\section{Preliminary results on vector bundle stability}
\label{prelim}

Given a vector bundle $\cE$ on $C$ its slope is the rational number $\mu(\cE):=\deg \cE/\rk \cE$.
\begin{definition}
\label{mustabilita}
 The vector bundle $\cE$ is {\it stable}  (respectively {\it semistable}) if for any proper subbundle ${\cF}\subset \cE$, we have that $\mu(\cF)<\mu (\cE)$ (resp. $\leqslant$).
\end{definition}

Throughout the paper we will consider the following setting.

%

Let  $M_{V,{\cal L}} = \ker (V \otimes \oo \twoheadrightarrow {\cal L})$ the associated dual span bundle.
Let ${\cal S} \subset M_{V,{\cal L}}$ be  a  saturated proper subbundle.
Then there exist a vector bundle $F_{\cal S}$ and a subspace $W \hookrightarrow V$ fitting into the commutative diagram
$$
\xymatrix{
0\ar[r]&{\cal S}\ar[r]\ar@{^{(}->}[d]&W\otimes \oo_C\ar[r] \ar@{^{(}->}[d] &F_{\cal S}\ar[r]\ar[d]^\alpha& 0\\
0\ar[r]&M_{V,{\cal L}}\ar[r]&V\otimes \oo_C\ar[r]  &{\cal L}\ar[r]& 0}
$$
Indeed, define  $W \hookrightarrow V$ by
$W^* := \mathrm{Im} (V^* \to H^0({\cal S}^*))$;
then  $W^*$ generates ${\cal S}^*$. 
Then define $F_{\cal S}^*:= \ker (W^* \otimes \oo \twoheadrightarrow {\cal S}^*)$.

\begin{remark}\label{conticini}
Let us summarize some properties of these objects, well known to experts; see for instance \cite{ButNormal} for reference.
With the notation above,  the following hold.
\begin{enumerate}
\item The sheaf $F_{\cal S}$ is globally generated and $h^0(F_{\cal S}^*)=0$.
\item The induced map $\alpha\colon F_{\cal S}\longrightarrow {\cal L}$ is not the zero map.
\item If ${\cal S}$ is a maximal destabilizing for $M_{V,{\cal L}}$ then $\deg {F_{\cal S}}\leqslant \deg {\cal I}$, where ${\cal I}$ is $\mathrm{Im} (\alpha)$, and equality holds if and only if $\rk {F_{\cal S}}=1$.
\end{enumerate}
The only point worth verifying is the last.
We can form the following diagram

$$
\xymatrix{
0\ar[r] & {\cal S}\ar[r]\ar@{^{(}->}[d] & \ar[r]\ar@{=}[d]  W\otimes \oo_{{C}}\ar[r] & {F_{\cal S}}\ar[r]\ar[d]& 0 \\
 0\ar[r]& M_{W,I} \ar[r]\ar@{^{(}->}[d]& W \otimes \oo_{{C}}\ar[r]\ar[d] &{\cal I}\ar[r]\ar@{^{(}->}[d]  & 0 \\
0 \ar[r] & M_{V,{\cal L}} \ar[r]& V \otimes \oo_{{C}} \ar[r] & {\cal L} \ar[r] & 0\\
}
$$
If we require maximality of the subbundle ${\cal S}$,
and destabilization, 
we have 
\[
\mu ({\cal S}) =  \frac{- \deg {F_{\cal S}}}{\dim W - \rk {F_{\cal S}}} \geqslant \frac{-\deg {\cal L}}{\dim V -1} = \mu(M_{{\cal L}})\geqslant  \mu(M_{W, {\cal I}})= \frac{-\deg {\cal I}}{\dim W -1}.
\]
So, if $\rk F_{\cal S}>1$, we have
\[
\deg {F_{\cal S}} \leqslant \frac{\dim W -\rk F_{\cal S}}{\dim W -1} \deg {\cal I} < \deg {\cal I}.
\]
\end{remark}

\bigskip

\section{Linear stability and Clifford index}
\label{linear stability}

We give here a natural generalization of the notion of linear stability of a curve and a linear series on it, introduced by Mumford in \cite{Mum} (cf. \cite{LS}).

\begin{definition}\label{ls}
Let ${\cal L}$ be a degree $d$ line bundle on $C$, and $V\subseteq H^0({\cal L})$ a generating subspace of dimension $r+1$.
We say that the triple $(C, {\cal L},V)$ is {\em linearly semistable (resp. stable)} if  any linear series of degree $d'$ and dimension $r'$ contained in  $|V|$  satisfies $d'/r'\geqslant d/r$ (resp. $>$).
\end{definition}

In case $V=H^0({\cal L})$, we shall talk of the stability of the couple $(C, {\cal L})$. It is easy to see that in this case it is sufficient to verify that the inequality of the definition holds for any {\em complete} linear series in $|V|$.

\begin{remark}\label{ossLS}
It is clear that  the following conditions are equivalent:
\begin{enumerate}
\item the triple $(C,{\cal L}, V)$ is linearly stable;
\item the bundle $M_{V,{\cal L}}$ is not destabilized by any bundle $M_{V', {\cal L}'}$, with $V'\subseteq V$, and $V'\otimes \oo_C\longepi {\cal L}' \subset {\cal L}$.
\end{enumerate}
\end{remark}

Using the theorem of Riemann-Roch, it is not hard to prove that $(C,{\cal L})$ is linearly stable for any line bundle ${\cal L}$ of degree $\geqslant 2g+1$.

\medskip

Let now $C$ be of genus $g\geq 2$.
We now present a more general result relating linear stability to the Clifford index of the curve.
Let us recall that  {\em the Clifford index} of a curve $C$
of genus $g \geqslant 4$ is the integer: 
\[
\textrm{Cliff} (C) := 
\min \{ \deg ({\cal L}) -2 (h^0 ({\cal L}) -1) ~|~ {\cal L} \in \mbox{Pic}(C) ~,~ 
h^0({\cal L}) \geqslant 2 ~,~ h^1 ({\cal L}) \geqslant 2 \}.
\]

When $g = 2$,  we set $\textrm{Cliff} (C ) = 0$; when $g = 3$ we set $\textrm{Cliff} (C ) = 0$ or $1$ according to whether 
C is hyperelliptic or not.
%

Let $\gamma (C)$ be the gonality of the curve $C$. The following inequalities hold:
$$\gamma (C) -3 \leqslant \textrm{Cliff} (C) \leqslant \gamma (C) -2,$$
the case $\textrm{Cliff} (C) = \gamma (C) -2$ holding for general $\gamma (C)$-gonal curves in the moduli space of smooth curves
${\cal M}_g$.
Furthermore $\textrm{Cliff} (C) = 0$ if and only if $C$ is hyperelliptic.
\medskip

\begin{proposition}\label{linear}
Let $C$ be a curve of genus $g\geq 2$.
Let ${\cal L} \in \picard(C)$ be a globally generated line bundle
such that $\deg {\cal L} - 2 (h^0({\cal L}) -1) \leqslant \textrm{Cliff} (C)$.
Then ${\cal L}$ is linearly semistable. 
It is linearly stable unless 
${\cal L} \cong \omega_C(D)$ with $D$ an effective divisor of degree 2,
or
$C$ is hyperelliptic and  $\deg {\cal L} = 2 (h^0({\cal L}) -1)$.
\end{proposition}

\begin{proof}
Recall that it is sufficient to check linear stability for complete linear subsystems of $|{\cal L}|$ 
(cf. Remark \ref{ossLS}).
Let ${\cal P}\hookrightarrow {\cal L}$ be a line bundle generated by a subspace of $H^0({\cal L})$. 
Observe that  
$$H^1({\cal P})^* = H^0(\omega \otimes {\cal P}^*) 
\supseteq H^0(\omega \otimes {\cal L}^*) =  H^1({\cal L})^*.$$
Let us distinguish three cases:

(1) $h^1({\cal L})\geqslant 2$. In this case  ${\cal L}$ computes the Clifford index: 
$$\deg {\cal L}=2(h^0({\cal L})-1)-\textrm{Cliff} (C).$$
Hence $h^1({\cal P}) \geqslant h^1({\cal L}) \geqslant 2$ and ${\cal P}$ contributes 
to the Clifford index $\textrm{Cliff}(C)$, 
so  $\deg {\cal P} \geqslant 2(h^0({\cal L}) -1) + \textrm{Cliff}(C)$.

Then
\[
\frac{\deg {\cal P}}{\dim V -1} \geqslant \frac{\deg {\cal P}}{h^0({\cal P}) -1}
\geqslant 2 + \frac{\textrm{Cliff}(C)}{h^0({\cal P}) -1}
\geqslant 2 + \frac{\textrm{Cliff}(C)}{h^0({\cal L}) -1}
= \frac{\deg {\cal L}}{h^0({\cal L}) -1},
\]
where the last inequality is strict unless $\textrm{Cliff}(C) =0$ and 
$\deg {\cal L} = 2 (h^0({\cal L}) -1)$,
in which case the curve is hyperelliptic,
and  ${\cal L}$ is linearly semistable but not linearly stable
(it can be shown that the dual of the $g^1_2$ maps to $M_{\cal L}$ in this case).

(2) If $h^1({\cal L})=1$,
then either $h^1({\cal P}) =1$ or $h^1({\cal P}) \geqslant 2$.
In the last case ${\cal P}$ contributes to the Clifford index,
so 
\[
\deg {\cal P} / (h^0({\cal P})-1) \geqslant 2 + \textrm{Cliff} (C)/ (h^0({\cal L})-1) \geqslant \deg {\cal L} / (h^0({\cal L}) -1) ~,
\]
with strict inequality unless $C$ is hyperelliptic and 
$\deg {\cal L} = 2 (h^0({\cal L})-1)$.

If $h^1({\cal P}) = h^1({\cal L}) =1$,
then, as $\deg {\cal P} < \deg {\cal L}$, we have that
$\deg {\cal P} / (\deg {\cal P} + 1 - g) > \deg {\cal L} / (\deg {\cal L} + 1 - g)$.

(3) If $h^1({\cal L}) = 0$,
then $h^1({\cal P}) =0$, or $h^1({\cal P}) =1$, or $h^1({\cal P}) \geqslant 2$.
In the last case ${\cal P}$ contributes to the Clifford index, and we can reason as above.
If $h^1({\cal P}) = h^1({\cal L}) =0$,
then as $\deg {\cal P} < \deg {\cal L}$ we have
$\deg {\cal P} / (\deg {\cal P}  - g) > \deg {\cal L} / (\deg {\cal L}  - g)$.

At last, suppose that  $h^1 ({\cal P}) = 1$. Then of course $\deg {\cal P} \leqslant 2g -2$. Consider the exact sequence
\[
0 \to H^0({\cal P}) \to H^0({\cal L}) \to H^0(D, \oo_D) \to H^1({\cal P}) \to 0,
\]
where $D$ is an effective divisor  such that ${\cal  P}(D)\cong{\cal L}$.
From this sequence,  remarking that the inclusion $H^0({\cal P}) \subset H^0({\cal L})$ must be strict, we deduce that 
  $\deg {\cal L} - \deg {\cal P} = h^0(D, \oo_D) \geqslant 2$. 

We thus have the following chain of inequalities
\[
\frac{\deg {\cal P}}{h^0({\cal P})-1} = \frac{\deg {\cal P}}{\deg {\cal P} + 1 -g}  
\geqslant
\frac{\deg {\cal L}}{\deg {\cal L} -g}
= \frac{\deg {{\cal P}} + h^0(D, \oo_D)}{\deg {\cal P} + h^0({D}, \oo_D) -g}.
\]
In fact 
${\deg {\cal P}}/{(\deg {\cal P} + 1 -g)} \leqslant 
(\deg {{\cal P}} + h^0(D, \oo_D))/{(\deg {\cal P} + h^0({D}, \oo_D) -g)}$
if and only if $\deg {{\cal P}} \leqslant (h^0(D, \oo_D)) (g-1)$ and
as $\deg {\cal P} \leqslant 2g -2$
the inequality is always verified and is strict unless 
$\deg {\cal P} = 2g -2$ and $h^0(D, \oo_D) =2$.
In this last case we have that $h^0({\cal P}^* \otimes \omega_C)=1$
so ${\cal P} \cong \omega_C$ and ${\cal L} \cong \omega_C(D)$ as wanted.
\end{proof}

\begin{remark}
A similar result on non-complete canonical systems was obtained in \cite{barstop}: it states that the triple $(C, \omega_C, V)$ where $V$ is a general 
subspace $V \subset H^0(\omega_C)$ of codimension $c \leqslant \cliff (C) /2 ~$
is linearly semistable. Note that the condition on codimension is analogous to the condition of Proposition  \ref{linear}: $\deg \omega_C=2g-2 \leqslant 2(\dim V-1) - \cliff (C)$. 
\end{remark}

\section{The slope of determinant bundles}
\label{cliff}

Let us state  the following well known fact (see for instance 5.0.1 \cite{Huy}).

\begin{proposition}\label{detsequence}
Let ${\cal F}$ be a globally generated vector bundle of rank $r\geqslant 2$.
Let ${\cal A} = \det ({\cal F})$. For a general choice of a subspace $T\subset H^0({\cal F})$ of dimension $r-1$,
evaluation on global sections of ${\cal F}$ gives the following exact sequence: 
\begin{equation}
\label{seqdet}
0 \longrightarrow  T\otimes \oo_C \longrightarrow   {\cal F}  \longrightarrow   {\cal A}  \longrightarrow   0.
\end{equation}
\end{proposition}

The following argument will be a key point in our proof. 
It is largely inspired by \cite{Butler}. 

\begin{proposition}
\label{splitglob}
Let ${\cal F}$ be a globally generated vector bundle of rank $r\geqslant 2$ and  $h^0( {\cal F}^*)=0$.
If the sequence (\ref{seqdet}) is \emph{exact on global sections},
then  $\deg {\cal A} = \deg {\cal F} \geqslant \gamma (h^0({\cal A}) -1)$, where 
$\gamma$ is the gonality of the curve $C$.
\end{proposition}
\begin{proof}
Let us consider the sequence (\ref{seqdet}) tensored with $\omega_C$. By taking the cohomology sequence, as $h^0({\cal F}^*)=0$, 
we can conclude that the homomorphism $H^0({\cal F} \otimes \omega_C)\longrightarrow H^0({\cal A}\otimes \omega_C)$ is not surjective. 
From this, we derive that the multiplication homomorphism 
\begin{equation}\label{multiplication}
H^0({\cal A})\otimes H^0(\omega_C)\longrightarrow H^0({\cal A} \otimes \omega_C)
\end{equation}
fails to be surjective. 
Indeed, let us consider the commutative diagram:
$$
\xymatrix{
&& 0 \ar[d]& 0 \ar[d]&\\
&0\ar[r]\ar[d] &  \oplus^{\rk {\cal F} - 1} \omega_C\ar[d]\ar[r]& \oplus^{\rk {\cal F} - 1} \omega_C\ar[r]\ar[d]& 0\\
0\ar[r]&M_{{\cal F}}\otimes\omega_C\ar[r]\ar@{^{(}->}[d] &H^0({\cal F}) \otimes\omega_C\ar[r] \ar[d] &{\cal F}\otimes\omega_C\ar[r]\ar[d]& 0\\
0\ar[r]&M_{{\cal A} }\otimes\omega_C\ar[r] &H^0({\cal A} ) \otimes\omega_C\ar[r] \ar[d] &{\cal A} \otimes\omega_C\ar[r]\ar[d]& 0\\
&& 0&0&\\
}
$$
Remark that the middle column is exact by our assumptions on the exact sequence 
(\ref{seqdet}).
By taking global sections, we  have the commutative diagram
$$
\xymatrix{
H^0({\cal F})\otimes H^0(\omega_C)\ar[r]\ar[d]&H^0({\cal F} \otimes \omega_C)\ar[d]\\
H^0({\cal A} )\otimes H^0(\omega_C)\ar[r]&H^0({\cal A} \otimes \omega_C)\\
}
$$
where the first vertical arrow is surjective, while the second - as it is shown above - is not. 
Hence the bottom horizontal arrow cannot be surjective.

From a result of Castelnuovo type due to Mark Green (\cite{green} Theorem 4.b2, see also \cite{Ciliberto}) we have that,
for any base point free line bundle ${\cal A} $,
the sequence (\ref{multiplication}) fails to be surjective  only if the image of the morphism induced by ${\cal A} $ is a rational normal curve in $\pr(H^0({\cal A} )^*)$.
Hence we have that $\deg {\cal A} \geqslant \gamma (h^0({\cal A} )-1)$, where $\gamma$ is the gonality of $C$, as wanted. 

\end{proof}

We now state a consequence on dual span bundles that will be a key point in our arguments.
As usual, let ${\cal L}$ is a line bundle on  $C$ and $V\subseteq H^0({\cal L})$ a generating subspace. 
Let ${\cal S} \subset M_{V,{\cal L}}$ be  a  saturated subbundle, and $F_{\cal S}$ and ${\cal A}= \det F_{\cal S}$ as in   Remark \ref{conticini}.

\begin{lemma}
\label{ohi}
Suppose that  $\rk F_{\cal S}\geqslant 2$. If $F_{\cal S}$ fits in an exact sequence 
$$
0 \longrightarrow  \oplus^{\rk F_{\cal S} -1} \oo_C \longrightarrow   F_{\cal S}  \longrightarrow   {\cal A}  \longrightarrow   0 
$$
which is also exact on global sections,
then the following hold.
\begin{enumerate}
\item If ${\cal L}$ verifies $ \deg {\cal L} \leqslant \gamma (\dim V -1)$,
then $\mu ({\cal S}) \leqslant \mu (M_{V,{\cal L}})$.
Furthermore, we have equality if and only if
\begin{itemize}
\item[-] $W = H^0 (F_{\cal S})$,  
\item[-] $ \gamma = \deg {\cal A} /  (h^0({\cal A}) -1)$,
\item[-] $\gamma  = \deg {\cal L} /(\dim V -1)$.
\end{itemize}

\item If ${\cal L}$ verifies
\(
\deg {\cal L} < \gamma (\dim V -1)
\),
then $\mu ({\cal S}) < \mu (M_{V,{\cal L}})$.
\end{enumerate}
\end{lemma}
\begin{proof}
Note that, as $\rk {F_{\cal S}} \geqslant 2$, 
\[
\rk {\cal S} = \dim W - \rk {F_{\cal S}} 
\leqslant h^0 ({F_{\cal S}}) - \rk {F_{\cal S}} = h^0 ({\cal A}) -1.
\]
So if we have 
\[
\mu ({\cal S}) = \frac{-deg {\cal A}}{\dim W - \rk {F_{\cal S}}} \geqslant \mu (M_{V,{\cal L}}) 
= \frac{-deg {\cal L}}{\dim V - 1} ,
\]
then
\[
\gamma \leqslant \frac{\deg {\cal A}}{h^0({\cal A}) -1}  
\leqslant \frac{\deg {\cal A}}{\dim W - \rk {F_{\cal S}}} 
\leqslant \frac{\deg {\cal L}}{\dim V -1} \leqslant \gamma  .
\]

So the inequality 
$\mu ({\cal S}) \geqslant \mu (M_{V,{\cal L}})$
cannot hold strict,
and it is an equality if and only if
$W = H^0 ({F_{\cal S}})$, and  $ \gamma = \deg {\cal A} /  (h^0({\cal A}) -1)
 = \deg {\cal L} /(\dim V -1)$.
\end{proof}

\medskip


\section{Stability of DSB's in the complete case}
\label{complete}

The main result of this section is the first part of Theorem \ref{main}.

\begin{theorem}
\label{princthm}
Let ${\cal L} \in \picard(C)$ be a globally generated line bundle
such that 
$$\deg {\cal L} - 2 (h^0({\cal L}) -1) \leqslant \textrm{Cliff} (C).$$
Then ${\cal L}$ is linearly (semi)stable if and only if $M_{\cal L}$ is (semi)stable. 
\end{theorem}

\begin{proof}

Clearly $M_{\cal L}$  (semi)stable implies 
${\cal L}$ is linearly (semi)stable. Let us prove the other implication,
thus suppose $\cal L$ linearly (semi)stable.

By contradiction let ${\cal S}$ be a maximal stable destibilizing subbundle of $M_{{\cal L}}$,
\emph{i.e.} ${\cal S}$ stable, $\mu ({\cal S}) \geqslant \mu (M_{{\cal L}})$ maximal 
($>$ for semistability), 
and $\rk {\cal S} < \rk M_{{\cal L}}$.
Note that 
\begin{equation}
\label{elgam}
\deg {\cal L}  \leqslant 
\textrm{Cliff} (C) + 2 (h^0({\cal L})-1) 
 \leqslant \gamma -2 + 2 (h^0({\cal L})-1) 
 \leqslant \gamma (h^0({\cal L}) -1),
 \end{equation}
 with equality iff either $\gamma =2$ and $\deg {\cal L} = 2( h^0({\cal L}) -1)$,
or $h^0({\cal L}) = 2$ and $\deg {\cal L} = \gamma = \textrm{Cliff} (C) + 2$.

By the assumption on linear (semi)stability, we have that that $\rk {F_{\cal S}} \geqslant 2$.
We prove the following

{\bf Claim:}  The bundle ${F_{\cal S}}$ admits  a determinant  sequence (\ref{seqdet}) exact on global sections.

\noindent
Then --by (\ref{elgam}) and by (ii) of Remark 
\ref{conticini}-- we can apply Lemma  \ref{ohi}.
So for such ${\cal S}$ and $F_{\cal S}$ we have that ${\cal S}$ cannot destabilize $M_{\cal L}$.
It can  strictly destabilize (\emph{i.e.} $\mu ({\cal S}) = \mu (M_{{\cal L}})$)
only in the case where $\deg {\cal L} = \gamma ( h^0({\cal L}) -1)$.
By the consequences of (\ref{elgam}), this strict destabilization can happen only
if either $\gamma =2$ and $\deg {\cal L} = 2( h^0({\cal L}) -1)$,
or $h^0({\cal L}) = 2$ and $\deg {\cal L} = \gamma = \textrm{Cliff} (C) + 2$.
In the last case we have that $\rk M_{\cal L} =1$. In the first one, we have that $C$ is hyperelliptic
and $\deg {\cal L} = 2( h^0({\cal L}) -1)$;  it is well known that ${\cal L}$ is strictly 
linearly semistable in this case
(the dual of the $g^1_2$ providing a strict destabilization as noted above).
In any case we cannot have a strict destabilization if $\cal L$ is supposed linearly 
\emph{stable}.

To prove the {\bf claim} let us remark that
by Proposition \ref{detsequence} such a short exact sequence exists. What 
we need to show is that it is exact on global sections; this is equivalent to showing 
 that $h^0 ({\cal A}) \leqslant h^0({F_{\cal S}}) - \rk {F_{\cal S}} +1$.

Observe that
$H^0({F_{\cal S}}) \twoheadrightarrow H^0({\cal A})$ 
if and only if 
$H^1(\oo^{\oplus \rk {F_{\cal S}} -1}) \hookrightarrow H^1({F_{\cal S}})$. 
Let us show first that this is numerically possible:
we prove that $g(\rk {F_{\cal S}} -1) = h^1(\oo^{\oplus \rk {F_{\cal S}} -1}) < h^1 ({F_{\cal S}})$ indeed.
In fact
\[
h^1 ({F_{\cal S}}) 
= h^0({F_{\cal S}}) - \deg {F_{\cal S}} + g \cdot \rk {F_{\cal S}} - \rk {F_{\cal S}}.
\]
Hence $h^1({F_{\cal S}}) > g \cdot \rk {F_{\cal S}} -g$ if and only if
$h^0 ({F_{\cal S}}) - \rk {F_{\cal S}} > \deg {F_{\cal S}} -g$.

As 
$h^0 ({F_{\cal S}}) - \rk {F_{\cal S}} \geqslant \rk {\cal S}$,
we can show that $\rk {\cal S} > \deg {F_{\cal S}} -g$,
\emph{i.e.} that 
\[
\frac{\deg {F_{\cal S}}}{ \rk {\cal S}} < 1 + \frac{g}{\rk {\cal S}} .
\]
By hypothesis 
$\mu ({\cal S}) = - \deg {F_{\cal S}} / \rk {\cal S} \geqslant  - \deg {\cal L} / (h^0 ({\cal L}) -1)$,
hence
\[
\frac{\deg {F_{\cal S}} }{ \rk {\cal S} } \leqslant
\frac{ \deg {\cal L} }{ h^0({\cal L}) -1} 
=\frac{ \deg {\cal L} }{ h^1({\cal L}) + \deg {\cal L} -g} 
= 1 + \frac{g - h^1 ({\cal L})}{\rk M_{{\cal L}}} < 1 + \frac{g}{\rk {\cal S}}.
\]

As the cokernel of 
$\varphi \colon H^1(\oo^{\oplus \rk {F_{\cal S}} -1}) \longrightarrow H^1({F_{\cal S}})$
is exactly $H^1({\cal A})$,
and the inequality above is strict,
then if $h^1 ({\cal A}) \leqslant 1$ the map
$\varphi$ is injective as we need, 
and $H^0({F_{\cal S}}) \twoheadrightarrow H^0({\cal A})$.

Let us show that if $h^1 ({\cal A}) \geqslant 2$, then the map
is surjective as well:
in this case we have the inequality
\[
\deg {\cal A} -2(h^0({\cal A}) -1) \geqslant 
\textrm{Cliff} (C)
\geqslant \deg ({\cal L}) -2 (h^0 ({\cal L}) -1).
\]

As $\deg {F_{\cal S}} = \deg {\cal A} < \deg {\cal L}$ (see Remark \ref{conticini}),
then $2(h^0({\cal L}) - h^0({\cal A})) \geqslant \deg {\cal L} - \deg {\cal A} >0$,
hence $h^0({\cal A}) < h^0({\cal L})$.

Remark that by the assumption made on ${\cal S}$, 
\[
\frac{\deg {\cal A}}{\rk {\cal S}} = \frac{-\deg {\cal S}}{\rk {\cal S}} \leqslant \frac{\deg {\cal L}}{h^0 ({\cal L}) -1},
\] 
hence $\rk {\cal S} \geqslant \deg {\cal A} \cdot (h^0({\cal L}) -1) / \deg {\cal L}$.

Now assume that $H^0({F_{\cal S}}) \to H^0 ({\cal A})$ is not surjective,
\emph{i.e.} that 
$h^0 ({\cal A}) > h^0({F_{\cal S}}) - \rk {F_{\cal S}} +1$.

Then we have that 
\(
h^0({\cal A}) -1 > h^0({F_{\cal S}}) - \rk {F_{\cal S}} \geqslant \rk {\cal S} \geqslant \deg {\cal A} \cdot (h^0({\cal L}) -1) / \deg {\cal L}
\),
hence
\[
\deg {\cal A} < \frac{\deg {\cal L}}{h^0({\cal L}) -1} (h^0 ({\cal A}) -1) 
\leqslant
(2 + \frac{\textrm{Cliff} (C)}{h^0({\cal L})-1})(h^0({\cal A})-1) =
\]
\[
= 2(h^0({\cal A})-1) + \textrm{Cliff} (C) \frac{h^0({\cal A})-1}{h^0({\cal L})-1}
\leqslant 2(h^0({\cal A})-1) + \textrm{Cliff} (C) ~,
\]
so $2(h^0({\cal A})-1) + \textrm{Cliff} (C) \leqslant \deg {\cal A} 
< 2(h^0({\cal A})-1) + \textrm{Cliff} (C)$
and we get a contradiction.

\end{proof}

\begin{remark}
\label{incompl}
It is worth noticing that the claim in the proof of Theorem \ref{princthm} above is a point where Butler's argument in \cite{Butler} fails to be complete.
\end{remark}

The consequences of this theorem, as stated in the introduction, follow easily:

\begin{theorem}\label{paranjape}
Let ${\cal L} \in \picard(C)$ be a globally generated line bundle
such that 
\begin{equation}
\deg {{\cal L}} - 2 (h^0({\cal L}) -1 )  \leqslant \textrm{Cliff} (C).
\end{equation}
Then $M_{\cal L}$ is semistable, and it is strictly semistable only in one of the following cases

(i)
${\cal L} \cong \omega_C (D)$ with $D$ an effective divisor of degree 2,

(ii)
$C$ is hyperelliptic and $\deg {{\cal L}}= 2 (h^0({\cal L}) -1)$.
\end{theorem}

\begin{proof}
Follows immediately from Theorem
\ref{princthm} and Proposition \ref{linear}.

\end{proof}

\begin{corollary}\label{spiego1}
Let ${\cal L}$ be a  globally generated line bundle over $C$ with 
$$\deg {{\cal L}} \geqslant 2g - \cliff (C).$$
Then the vector bundle $M_{\cal L}$ is semistable. 
It is stable unless $(i)$ or $(ii)$ hold.
\end{corollary}

\begin{proof}

Observe that if  ${\cal L}$ is a  globally generated line bundle over $C$ with 
$\deg {{\cal L}} \geqslant 2g - \cliff (C),$ then
$\cliff (C) \geqslant 2g - \deg {{\cal L}}$, so
\[
\cliff C + 2 (h^0({\cal L}) -1) \geqslant 
\cliff C + 2 (\deg{\cal L} -g) \geqslant 
\deg {\cal L} ~,
\]
then use Theorem \ref{paranjape}.

\end{proof}

\begin{corollary}\label{spiego2}
Let ${\cal L}$ be any line bundle that computes the Clifford index of $C$. Then $M_{\cal L}$ is semistable; it is stable unless $C$ is hyperelliptic.
\end{corollary}

\begin{proof} Follows from Theorem \ref{paranjape}, recalling that any line bundle computing the Clifford index is globally generated.

\end{proof}

\section{The non-complete case}
\label{non-complete}

The aim of this section is to extend the methods described above, when possible, to the non-complete case.

The following conjecture is the most natural direct generalization of Theorem
\ref{princthm} to the non-complete case. Note that  it is weaker than Conjecture \ref{cong} below.
We will not prove it in full generality, but it still holds in many cases.

\begin{conjecture}
\label{weakconj}

Let $(C,{\cal L}, V)$ be a triple. 
If $\deg {\cal L} -  2  (\dim V-1) \leqslant \textrm{Cliff} (C)$,
then linear (semi)stability is equivalent to (semi)stability 
of $M_{V,{\cal L}}$.

\end{conjecture}

\begin{remark}

The inequality 
$\deg {\cal L}\leqslant  \textrm{Cliff} (C) + 2  (\dim V-1)$
holds if and only if
\[
\textrm{codim}_{H^0(L)} V \leqslant \frac{\textrm{Cliff} (C) - 
(\deg {\cal L} -2(h^0({\cal L})-1))}{2} ~.
\]

\end{remark}

The results of this and the previous section can be summarized in the following
(equivalent to Theorem \ref{main})

\begin{theorem}
\label{noncompl}

Conjecture \ref{weakconj} holds in the following cases:

\begin{enumerate}

\item ${H^0({\cal L})} = V $ (complete case);

\item $\deg {\cal L}  \leqslant 2g - \textrm{Cliff} (C) +1$;

\item  $\mathrm{codim}_{H^0({\cal L})} V < h^1({\cal L})+g/(\dim V-2)$;
 
 \item $\deg {\cal L} \geqslant 2g$, and $ \textrm{codim}_{H^0({\cal L})} V \leqslant (\deg {\cal L} -2g)/2 $.

\end{enumerate}

\end{theorem}

In all the following result we make this assumption. 
Let $(C,{\cal L}, V)$ be a triple verifying
$\deg {\cal L} -  2  (\dim V-1) \leqslant \textrm{Cliff} (C)$,
let  ${\cal S}\subset M_{V,{\cal L}} $ 
be a proper subbundle such that
 $ \mu ({\cal S}) \geqslant \mu (M_{V,{\cal L}}) $, let $F_{\cal S}$ and ${\cal A}$ be as in Lemma \ref{ohi}.

In order to prove Theorem \ref{noncompl},
we proceed as for Theorem \ref{princthm},
and show that within these numerical hypothesis we can apply 
Lemma \ref{ohi}. That is, we show that for a possible destabilization given by
$$
\xymatrix{
0\ar[r]&{\cal S}\ar[r]\ar@{^{(}->}[d]&W\otimes \oo_C\ar[r] \ar@{^{(}->}[d] &F_{\cal S}\ar[r]\ar[d]^\alpha& 0\\
0\ar[r]&M_{V,{\cal L}}\ar[r]&V\otimes \oo_C\ar[r]  &{\cal L}\ar[r]& 0}
$$
the bundle $F_{\cal S}$ fits into a short exact sequence
$$
0 \longrightarrow  \oplus^{\rk F_{\cal S} -1} \oo_C \longrightarrow   F_{\cal S}  \longrightarrow   {\cal A}  \longrightarrow   0 
$$
which is exact on global sections.

\begin{lemma}\label{lem1}
If $h^1({\cal A}) \geqslant 2$ then the exact sequence 
$$
0 \longrightarrow  \oplus^{\rk F_{\cal S} -1} \oo_C \longrightarrow   F_{\cal S}  \longrightarrow   {\cal A}  \longrightarrow   0 
$$
is exact on global sections.

\end{lemma}

\begin{proof}

If $h^1({\cal A}) \geqslant 2$ then
$\deg {\cal A} -2 (h^0({\cal A}) -1) \geqslant \cliff (C) \geqslant \deg {\cal L} -2 (\dim V -1)$, and then we have 
that $2(\dim V - h^0({\cal A})) \geqslant \deg {\cal L} - \deg {\cal A} = \deg {\cal L} - \deg F_{\cal S} >0$.
Furthermore, 
if 
\[
\frac{\deg {\cal A}}{\rk {\cal S}} = \frac{- \deg {\cal S}}{\rk {\cal S}} \leqslant \frac{\deg {\cal L}}{\dim V -1}   
\]
then $\rk {\cal S} \geqslant \deg {\cal A} (\dim V -1) / \deg {\cal L}$.

Then if we had the inequality $h^0({\cal A}) -1 > h^0(F_{\cal S}) - \rk F_{\cal S}$, we would have
$h^0({\cal A}) -1 > \rk {\cal S} \geqslant \deg {\cal A} (\dim V -1) / \deg {\cal L}$,
and then
\[
2(h^0({\cal A})-1) + \cliff (C) \leqslant \deg {\cal A} < \deg {\cal L} \frac{h^0({\cal A})-1}{\dim V -1}
\leqslant
\]
\[
\leqslant 2 (h^0({\cal A}) -1) + \cliff (C) \frac{h^0({\cal A})-1}{\dim V -1}
< 2 (h^0({\cal A}) -1) + \cliff (C),
\]
which is absurd, so we have 
$h^0({\cal A}) \leqslant h^0(F_{\cal S}) - \rk F_{\cal S} + 1$, hence the sequence
$$
0 \longrightarrow  \oplus^{\rk F_{\cal S} -1} \oo_C \longrightarrow   F_{\cal S}  \longrightarrow   {\cal A}  \longrightarrow   0 
$$
is exact on global sections.

\end{proof}

To complete the proof of Theorem \ref{noncompl} we have to treat the case  $h^1({\cal A}) \leqslant 1$ as well.

\begin{lemma}\label{lem2}

Suppose that  $h^1({\cal A})\leqslant 1 $. If we assume that
$\deg {\cal L}  \leqslant 2g - \textrm{Cliff} (C) +1$,
then the sequence
$$
0 \longrightarrow  \oplus^{\rk F_{\cal S} -1} \oo_C \longrightarrow   F_{\cal S}  \longrightarrow   {\cal A}  \longrightarrow   0 
$$
is exact on global sections.

\end{lemma}

\begin{proof}

We want to prove that $h^0(F_{\cal S}) - \rk F_{\cal S} + 1\geqslant h^0({\cal A})$.
This is the case if we prove that 
$h^0(F_{\cal S}) - \rk F_{\cal S} > \deg {\cal A} -g$.
As $h^0(F_{\cal S}) - \rk F_{\cal S} \geqslant \rk {\cal S}$,
recalling that $\deg F_{\cal S} = \deg {\cal A}$,
we are done if 
we can prove the following

{\bf Claim:}
$\deg F_{\cal S} < \rk {\cal S} + g$.

\noindent
In fact we have that
\[
\deg F_{\cal S} \leqslant \frac{\rk {\cal S}}{\dim V -1} \deg {\cal L} \leqslant
\frac{\rk {\cal S}}{\dim V -1} (\cliff (C) + 2 (\dim V -1)) =
\]
\[
=\cliff (C) \frac{\rk {\cal S}}{\dim V -1} + 2 \rk {\cal S}
< \cliff (C) + 2 \rk {\cal S}.
\]
So, if $\rk {\cal S} \leqslant g - \cliff (C)$ then the claim is verified.
Let us show that this is the case when 
$\rk {\cal S} > g - \cliff (C)$ as well.
In fact if we had $\deg F_{\cal S} \geqslant \rk {\cal S} + g$ holding together with
$\rk {\cal S} > g - \cliff (C)$, then we would have
$\deg {\cal L} > \deg F_{\cal S} > 2g - \cliff (C)$,
contrary to the assumption.

\end{proof}

\begin{lemma}\label{lem3}

If 
$\mathrm{codim}_{H^0({\cal L})} V < h^1({\cal L})+ g/(\dim V-2)$,
then the sequence
$$
0 \longrightarrow  \oplus^{\rk F_{\cal S} -1} \oo_C \longrightarrow   F_{\cal S}  \longrightarrow   {\cal A}  \longrightarrow   0 
$$
is exact on global sections.

\end{lemma}

\begin{proof}
The case $h^1({\cal A})\geqslant 2$ is treated in Lemma \ref{lem1}.
Let us assume that $h^1({\cal A})\leqslant 1$.
We proceed as in the proof of Lemma \ref{lem2},
and show the following 

{\bf Claim:}
$\rk {\cal S} > \deg {\cal A} - g$.

\noindent
As shown in Lemma \ref{lem2}, this implies that 
 $h^0(F_{\cal S}) - \rk F_{\cal S} + 1\geqslant h^0({\cal A})$.

To prove the claim, set $c : = \mathrm{codim}_{H^0({\cal L})} V$, and observe that 
\[
\rk {\cal S} > \deg {\cal A} -g 
\iff
\frac{\deg {\cal A}}{\rk {\cal S}} < 1 + \frac{g}{\rk {\cal S}}
\]
and that 
\[
\frac{\deg {\cal A}}{\rk {\cal S}} \leqslant \frac{\deg {\cal L}}{\dim V -1}
= \frac{\deg {\cal L}}{h^1({\cal L}) + \deg {\cal L} -g -c}
= 1 + \frac{g + c - h^1({\cal L})}{\dim V -1} 
 ~.
\]
Now observe that if $c< h^1({\cal L})+ g/(\dim V-2)$, we have, noting that $\rk {\cal S}\leqslant \dim V-2$,
\[
c-h^1({\cal L})< g\left( \frac{\dim V-1}{\rk{\cal S}}-1 \right),
\]
and hence that 
\[
\frac{\deg {\cal A}}{\rk {\cal S}}< 1+\frac{g}{\rk{\cal S}}.
\]

\end{proof}

As for the last point in Theorem \ref{noncompl}, it follows directly from
Lemma 2.2 in 
\cite{stabtrans}.

\begin{proposition}
\label{bs08}

Let $C$ be a curve such that $\cliff (C) \geqslant 4$.
Let $V \subset H^0(\omega_C)$ be a general subspace of codimension 
smaller than or equal to 2,  then $M_{V, \omega_C}$ is semistable.

\end{proposition}

\begin{proof}
In the complete case the (semi)stability is well known in the literature, regardless to the Clifford index
\cite{PR}.
In the non-complete case, it has been proven in \cite{barstop} that a general projection from a subspace of dimension smaller than or equal to $\cliff(C)/2$ is linearly stable.
Then, the proof is immediate from Theorem \ref{noncompl}.
\end{proof}

\begin{example}
From Lemma \ref{lem3} we can also construct a stable bundle with slope $-3$ as follows. 
Consider a genus $10$ curve $C$ with general Clifford index $4$. Let $V\subset H^0(\omega_C)$ be a $7-$dimensional subspace.
The assumptions of Lemma \ref{lem3} are thus satisfied, as 
$$\mathrm{codim}_{H^0(\omega_C)} V= 3 <  1+ \frac{5}{2}= h^1(\omega_C)+ \frac{g}{\dim V-2}. $$
If we choose $V$ to be general, we have that linear stability holds for $(C, \omega_C, V)$ by \cite{barstop}, and hence by Lemma \ref{lem3} the sheaf $M_{\omega_C, V}$ is stable.
Its slope is indeed $-\deg{\omega_C}/(\dim V-1)= -3$.
\end{example}


\section{Cohomological stability and the Clifford index}
\label{cohom}

The following definition was introduced by Ein and Lazarsfeld in \cite{EL}.

\begin{definition}\label{cohostab}
Let ${\cal E}$ be a vector bundle on a curve $C$. We say that ${\cal E}$ is {\em cohomologically stable} (resp. {\em cohomologically semistable}) if for any line bundle ${\cal A}$ of degree $a$ and for any integer $t<\rk {\cal E}$ 
$$h^0(\wedge^t{\cal E}\otimes {\cal A}^{-1})=0 \quad \mbox{ whenever  } \quad a\geqslant t \mu({\cal E}) \quad \mbox{(resp. $>$)} $$
\end{definition}

\begin{remark}\label{rapporto}
Cohomological (semi)stability implies  bundle (semi)stability; indeed, given any   proper subbundle ${\cal S}\subset {\cal E}$ of degree $a$ and rank $t$, we have an inclusion $\det {\cal S}\hookrightarrow \wedge^t{\cal E}$, hence a non-zero section of $(\det  {\cal S})^{-1}\otimes  \wedge^t{\cal E}$. 

Moreover, observe that cohomological (semi)stability of ${\cal E}$ is implied by $\wedge^t{\cal E}$ being (semi)stable for any integer $t$; hence cohomological semistability is equivalent to semistability, while cohomological stability can be a stronger condition than stability.
\end{remark}

In \cite{EL} the two authors prove the cohomological stability of the DSB $M_{\cal L}$ associated to any line bundle ${\cal L}$ on a curve of positive genus $g$, under the assumption that 
$\deg {\cal L}\geqslant 2g+1$. 

The main result of this section is Theorem \ref{uffa} stated in the Introduction, which is a generalization of the result of Ein and Lazarsfeld.

\begin{theorem}\label{uffa}
Let $({\cal L}, V)$ be a $g^r_d$ on a smooth curve $C$, inducing a birational morphism. Suppose that
\begin{itemize}
\item[-] $d\leqslant 2r+\cliff C$; 
\item[-] $\mathrm{codim}_{H^0({\cal L})} V \leqslant h^1({\cal L})$. 
\end{itemize}
Then $M_{V,{\cal L}}$ is cohomologically semistable. It is strictly stable unless $d=2r$. 
\end{theorem}

In order to prove Theorem \ref{uffa}, let us first establish this  simple generalization of a result used in the proof of Proposition 3.2 in \cite{EL}, and a lemma.

\begin{proposition}\label{proppopa}
Let $({\cal L}, V)$ be a $g^r_d$ on a smooth curve $C$, inducing a birational morphism; let $D_k=p_1+ \ldots+ p_k$ be a general effective divisor on $C$, with $k<r$.  
The DSB associated to the linear series  lies in the following exact sequence of sheaves
$$
0\longrightarrow M_{V(-D_k),{\cal L}(-D_k)}\longrightarrow M_{V, {\cal L}}\longrightarrow \oplus_{i=1}^k\oo_C(-p_i)\longrightarrow 0.
$$
\end{proposition}
\begin{proof}
As $D_k$ is general effective, we have that $\dim V(-D_k)= \dim V- k=r+1-k$. Let $W$ be the cokernel of the injection $V(-D_k)\subseteq V$.
Moreover, as the morphism induced by $|V|$ is generically injective, ${\cal L}(-D_k)$ is generated by $V(-D_k)$. 

Using the snake lemma, we can form the top exact row  in the diagram below, and the proof is concluded.
$$
\xymatrix{
&0\ar[d] & 0\ar[d]& 0\ar[d]& \\
0\ar[r]&M_{ V(-D_k), {\cal L}(-D_k)}\ar[r]\ar[d] &M_{V, {\cal L}}\ar[r]\ar[d] &\oplus_{i=1}^k\oo_C(-p_i)\ar[r]\ar[d]& 0\\
0\ar[r]& V(-D_k)\otimes\oo_C\ar[r] \ar[d] &V \otimes\oo_C\ar[r] \ar[d] &W\otimes \oo_C\ar[r]\ar[d]& 0\\
0\ar[r]&{\cal L}(-D_k)\ar[r]\ar[d]&{\cal L}\ar[r]\ar[d]&{\cal L}_{D_k}\ar[r]\ar[d]&\\
&0& 0&0&\\
}
$$
\end{proof}

\begin{remark}\label{rempopa}
With the notations and conditions of the above proposition, if we consider a general effective divisor  $D$ of maximal degree $r-1$, we have that 
$M_{V(-D),{\cal L}(-D)}$ is a line bundle which is dual to $\oo_C(p_r+\ldots + p_d)$, so 
$$M_{V(-D),{\cal L}(-D)}\cong \oo_C(-p_r-\ldots - p_d).$$
\end{remark}


\begin{lemma}
\label{lemstop}

Let ${\cal A}$ be a line bundle on the curve  $C$ such that $\deg {\cal A} \leqslant t d/r$, with $ t$, $d$, and $r$ integers verifying 
\(
0 < t < r < d \leqslant 2r + \mathrm{Cliff} (C)  \textrm{ and } 
r \geqslant d-g
\).
Suppose that the first inequality  $\deg {\cal A} \leqslant t d/r$ is strict in case 
$\mathrm{Cliff} (C) =0 $ and $d/r =2$.

Then $h^0({\cal A}) \leqslant t$.

\end{lemma}

\begin{proof}

Let us distinguish three cases according to the values of $h^1({\cal A})$:
\begin{itemize}
\item[(A)] Suppose that $ h^1({\cal A})\geqslant 2$. 
Then we can suppose that ${\cal A}$ contributes to the Clifford index of $C$ (i.e. that $h^0({\cal A})\geqslant 2$), and so 
$$
\begin{aligned}
2(h^0({\cal A})-1)\leqslant \deg {\cal A} -\cliff C
\leqslant t\frac{d}{r}-\cliff C\leqslant \\
\leqslant t\left(2+ \frac{\cliff C}{r}\right) -\cliff C
= 2t + \cliff C\left(\frac{t}{r}-1\right).
\end{aligned}
$$
The last quantity is strictly smaller than $2t$ if and only if $\cliff C>0$. So, if $C$ is non-hyperelliptic we are done. 
If $C$ is hyperelliptic we still have the claim if $d/r < 2$, 
while if $d/r=2$, 
the claim is true supposing the strict inequality $\deg {\cal A}< td/r$.

\item[(B)] Suppose that ${\cal A}$ is non special: $h^1({\cal A})= 0$. Then by Riemann-Roch $h^0({\cal A})= \deg {\cal A} -g+1$.
This quantity is smaller or equal to $t$ if and only if $\deg {\cal A} -g<t$. 
Hence, it is sufficient to prove that $td/r-g<t$; equivalently, we need to prove that 
$$t<r\frac{g}{d-r}.$$
By assumption, we have that $r\geqslant d-g$, so the above inequality is true, being $t<r\leqslant rg/(d-r)$.

\item[(C)] Eventually, let us suppose that $h^1({\cal A})= 1$. Remember that we are assuming that $d\leqslant 2r +\cliff C$. Let us distinguish three cases again.
\begin{itemize}
\item[(C.1)] If $d<2r$ (for instance this is the case if $d >2g$ as in \cite{EL}). 
Then $\deg {\cal A} \leqslant td/r<2t$. 
So, as ${\cal A}$ is special, we have that 
$2(h^0({\cal A})-1) \leqslant \deg {\cal A} <2t$ and we are done.
\item[(C.2)] Suppose that $d>2r$. 
Observe that $$ r\left(\frac{g-1}{d-r}\right)= r-\frac{r}{d-r}> r-1\geqslant t.$$
As in point (B), this is the inequality we need.
\item[(C.3)] Let us now suppose that $d=2r$. In this case, we obtain  $r(g-1)/(d-r)\geqslant t$, and we can prove the claim if we assume that  $\deg {\cal A}< td/r$.
\end{itemize}
\end{itemize}

\end{proof}

\begin{proof}[Proof of Theorem \ref{uffa}]
By Proposition \ref{proppopa} and Remark \ref{rempopa}, we have that the bundle $M_{V,{\cal L}}$ sits in the exact sequence
$$
0\longrightarrow  \oo_C(-p_r-\ldots - p_d) \longrightarrow M_{V, {\cal L}}\longrightarrow \oplus_{i=1}^r\oo_C(-p_i)\longrightarrow 0.
$$
Let $t$ be an integer strictly smaller than $r$.
Applying the $t$-th exterior power, we get the sequence

\begin{equation}\label{esteriore}
\begin{aligned}
0\longrightarrow  
\bigoplus_{1\leq i_1<\ldots i_{t-1}\leq r} \oo_C(-p_{i_1}-p_{i_2}-\ldots -p_{i_{t-1}}-p_r-\ldots - p_d) 
\longrightarrow \\
\longrightarrow \bigwedge^t  M_{V, {\cal L}}\longrightarrow \bigoplus_{1\leq j_1<j_2\ldots <j_t\leq r}\oo_C(-p_{j_1}-p_{j_2}-\ldots-p_{j_t})\longrightarrow 0.
\end{aligned}
\end{equation}

Let us now tensor the above sequence with  a line bundle  ${\cal A}^{-1} $ of degree $-a$. 

We shall now suppose that $-a\leqslant td/r$, in order to prove cohomological stability. We will see in the course of the proof that in case $d/r=2$ we will need to assume, strict inequality thus proving semistability (and precisely strict semistability in this case, of course).

We want to prove that $H^0(\wedge^t M_{V, {\cal L}}\otimes {\cal A}^{-1})=\{0\}$. To this aim, as in \cite{EL}, let us consider the global sections of sequence (\ref{esteriore}) 
tensorised by ${\cal A}^{-1}$ and  prove that both left and right side are trivial.

The left hand side is a sum of global sections of line bundles each of degree $-t-d+r-a\leqslant -t-d+r+(td)/r= (r-t)(1-d/r)$.
As $r-t>0$ by assumption, and $d>r$, this degree is negative and we are done.

Let us now study the right hand side:
applying Lemma \ref{lemstop} to ${\cal A}^{-1}$, we have $h^0({\cal A}^{-1})\leq t$
(remark that the hypothesis on $h^1 ({\cal L})$ is equivalent to $r \geqslant d-g$).

By construction, the pieces of the right hand side are of the form 
$H^0({\cal A}^{-1}(-D))$ where $D$ is a general effective divisor of degree $t$, so they vanish. 

\end{proof}

As a consequence, we have this result:
\begin{corollary}
If $C$ is non-hyperelliptic, then $M_{\omega_C}$ is cohomologically stable. 
Moreover, if $\cliff C \geqslant 2$ , 
a projection from a general point in $\pr^{g-1}$ is cohomologically stable. 
\end{corollary}

%

\begin{remark}\label{pullback}
What can we say if we drop the assumption of the linear series to induce a birational morphism? 
Let $({\cal L}, V)$ be a base point free linear series on a curve $C$. Let $\varphi\colon C\longrightarrow \pr^r$ the induced morphism, and $\nu\colon \overline C\rightarrow \varphi(C)\subseteq \pr^r$ the normalization of the image curve. Then the morphism $\varphi$ decomposes as 
$$
\xymatrix{C\ar[r]^{\beta}&\overline C\ar[r]^{\nu}& \varphi (C)\ar@{^{(}->}[r]^{\iota}&\pr^r}
$$
where $\beta$ is a finite morphism (of degree $b$). 
Let $(\overline {\cal L}, \overline V)$  be the linear series induced on  $\overline C$ by $ \iota\circ \nu$.
Clearly $V=\beta^*(\overline V)$, and ${\cal L}=\beta^*\overline {\cal L}$, and $\deg \overline{\cal L}= (\deg{\cal L})/b$.

Proposition \ref{proppopa} still holds if we substitute $D_k$ with $\beta^*(\beta(D_k))$, and   the points of this divisor fail to impose independent conditions on $H^0({\cal L})$, so that the argument of Theorem \ref{uffa}  cannot be pushed through. 

Observe that $\beta^* M_{ \overline V,\overline{\cal L}}=M_{V,{\cal L}}$. If $(\overline{\cal L}, \overline V)$ satisfies the numerical conditions of Theorem \ref{uffa}, then  $M_{ \overline V,\overline{\cal L}}$ is cohomologically stable by Theorem 
\ref{uffa}, so its pullback $M_{V,{\cal L}}$ is semistable, but we cannot say anything about its cohomological stability, nor  vector bundle stability.

\smallskip

On the other hand, it is worth noticing that  linear stability is preserved by finite morphisms: it is easy to verify that $(C, {\cal L}, V)$  is linearly (semi)stable if and only if  
$(\overline C,  \overline{\cal L}, \overline V)$  is linearly (semi)stable.
\end{remark}

\section{Linear series of dimension 2 and counterexamples}
\label{sezserie2}

In this section we discuss linear stability for curves with a $g^2_d$,
then exhibit some examples and counter-examples to the implication (linear stability of the triple $(C, {\cal L}, V)$ $\Rightarrow$ stability of $M_{V,{\cal L}}$).

The first result shows that linear stability is in this case related to the singularity of the image.

\begin{proposition}\label{linstabpiane}
Let $\nu \colon C \longrightarrow \mathbb{P}^2$ be a birational morphism.
Call $\overline C\subset \pr^2$ its image, and $d$ the degree of $\overline C$ in $\pr^2$.
The morphism $\nu$ is induced by a linearly (semi)stable
linear system if and only if all points 
$p \in \overline C$ have multiplicity $m_p < d/2$ (or $m_p \leqslant d/2$ for semistability).
\end{proposition}

\begin{proof}
Linear stability (resp. semistability) is equivalent  to the fact that any projection from a point $p \in \mathbb{P}^2$,
has degree $> d/2$ (resp. $\geqslant d/2$). This degree is precisely $d-m_p$.
\end{proof}

From this result we can easily derive linear stability for any general $g^2_d$ contained in a very ample linear series ${\cal L}$: 
 such a linear series induces a birational morphism whose image in $\pr^2$ is an integral plane curve with at most nodes as singularities
(cf. \cite{ACGH} Exercises B-5 B-6).
Hence this series is linearly stable (resp. semistable) 
as soon as $d>4$  (resp $\geqslant 4$).
Summing up we have proven the following 

\begin{proposition}\label{serie2}
Let $C$ be a smooth curve with an embedding in $\pr^n$ of degree $d>4$  (respectively $\geqslant 4$). The general projection on $\pr^2$ is linearly stable (resp. semistable).
\end{proposition} 

Clearly this result goes in the direction of Butler's conjecture. 
It is not hard to prove the stability of DSB for smooth plane curves, and anytime the degree $d$ is greater or equal to $4g$.
However we don't know in the general case whether or not linear stability implies the stability of the associate DSB.

\medskip

We now describe an example showing that linear stability is not always equivalent to stability of DSB.
Let us start with the following easy lemma. 
\begin{lemma}\label{piuno}
Let $|V|$ be a base point-free linear series of dimension $r$ contained in $H^0(\pr^1, \oo_{\pr^1}(d))$. Then if $r\!\!\!\not| d$, the dual span bundle $M_{ V, \oo_{\pr^1}(d)}$ is unstable.
\end{lemma}
\begin{proof}
The bundle  $M_{ V, \oo_{\pr^1}(d)}=\ker (V \otimes \mathcal{O}_C \to \oo_{\pr^1}(d))$ 
is a rank $r$  vector bundle on $\pr^1$, that splits in the direct sum of $r$ line bundles.
If $r$ does not divide $d$,  this bundle cannot  be (semi)stable.
\end{proof}

Combining the above lemma with Proposition \ref{linstabpiane}, we easily get counterexamples, as follows.  
\begin{proposition}\label{contro1}
On any curve $C$ there exist non-complete linear systems $V\subset H^0( {\cal L})$ such that  $(C, {\cal L}, V)$ is linearly stable and $M_{V, {\cal L}}$ is unstable. 
\end{proposition}

\begin{proof}
Consider any finite morphism $\beta\colon C\longrightarrow \pr^1$, and choose a map $\eta \colon \pr^1\longrightarrow \pr^2$ associated to a general base point free $ W\subset H^0(\pr^1, \oo_{\pr^1}(d))$ with odd degree $d>4$. By Lemma \ref{piuno}  the bundle $M_{ W, \oo_{\pr^1}(d)}$ is unstable. Let ${\cal L}= \beta^* \oo_{\pr^1}(d)$ and $V:= \beta^*(W)\subset H^0({\cal L})$ the linear series associated to the composition $\eta\circ \beta$. Clearly also  $M_{V, {\cal L}}=\beta^*M_{ W, \oo_{\pr^1}(d)}$ is unstable. 
On the other hand, $(\pr^1, \oo_{\pr^1}(d), W)$ is linearly stable, and -as  linear stability  respect finite morphism (Remark \ref{pullback})- so is $(C,{\cal L}, V)$.
\end{proof}

\begin{remark}
Note that the linear systems produced satisfy the inequality
$$\deg {\cal {\cal L}} \geqslant  \gamma d > \gamma (\dim V-1),$$
where $\gamma$ is the gonality of $C$, so there is no contradiction with our conjectures.
Furthermore, the subspace $V\subset H^0({\cal {\cal L}})$ is not general.

Therefore it seems reasonable to formulate some conjectures respectively on
the non complete and complete case.

\end{remark}

\begin{conjecture}\label{cong}
Let $(C,{\cal L}, V)$ a triple as usual.  If $\deg {\cal L}\leqslant  \gamma (\dim V-1)$, where $\gamma $ is the gonality of $C$, then linear (semi)stability is quivalent to
 (semi)stability 
of $M_{V,{\cal L}}$.
\end{conjecture}

\begin{conjecture}
For any curve $C$, and any line bundle ${\cal L}$ on $C$, linear 
(semi)stability 
of $(C, {\cal L})$ is equivalent to (semi)stability of $M_{\cal L}$.
\end{conjecture}

These conjectures arise implicitly from Butler's article \cite{Butler}
(cf. Remark \ref{incompl}).

%


%


%



\section{Stable DSB's with slope 3, and their theta-divisors}
\label{theta}

In this section we construct explicitly some stable bundles of integral slope on a general curve, and we prove that they admit theta-divisors.

Let us consider a curve $C$ of even genus $g = 2k$ having general gonality
$\gamma = k+1$ and Clifford Index $\cliff (C) = k-1$.
Let $D$ be a gonal divisor: $h^0(D) = 2$ and $\deg D = k+1$, and hence 
$h^1 (D) = k$ from the Riemann-Roch formula.

Let ${\cal L} = \omega_C (-D)$, then we have that $\deg {\cal L} = 2g -2 -k -1 = 3k -3$,
$h^0({\cal L}) = h^1(D) = k$,
$h^1({\cal L}) = h^0(D) = 2$, 
so $\deg {\cal L} - 2 (h^0({\cal L}) -1) = \cliff (C)$ 
and ${\cal L}$ computes the Clifford Index of $C$.

So the dual span bundle $M_{\cal L}$ is stable, and has integral slope
\[
\mu (M_{\cal L}) = - \frac{3k -3}{k-1} = -3
\]

\begin{question}

Does $M_{\cal L}$ admit a theta-divisor?

\end{question}

We recall that the vector bundle ${\cal E}$ with integral slope is said to admit a theta-divisor if
\[
\Theta_{{\cal E}} = \{ {\cal P} \in \picard^{g-1 - \mu({\cal E})} (C) ~|~ h^0 ({\cal P} \otimes {\cal E}) \neq 0  \} 
\subsetneq \picard^{g-1 -\mu({\cal E})} (C) ~.
\]
If this is the case, then 
$\Theta_{{\cal E}}$ has a natural structure of (possibly non reduced) 
divisor in $\picard^{g+2} (C)$, whose cohomology class is 
$r \cdot \vartheta$ where $r$ is the rank of ${\cal E}$ and $\vartheta$ is the class of the canonical theta-divisor in 
$\picard^{g-1-\mu({\cal E})} (C)$.

A vector bundle admitting a theta-divisor is semistable,
and if the vector bundle admits a theta divisor  
and is strictly semistable then the theta-divisor is not integral
(cf. \cite{DN}, and \cite{bove}).


\begin{proposition}
If $C$ is general, the vector bundle  $M_{\cal L}$ constructed above admits a theta-divisor.
\end{proposition}

\begin{proof}

Recall that the genus of the curve is $g = 2k$,
$\deg {\cal L} = 3k -3$,
and $\mu (M_{\cal L}) = -3$.
So $M_{\cal L}$ admits a theta-divisor if 
there exists a line bundle ${\cal P}$ of degree $\deg {\cal P} = g+2$ such that
$h^0({\cal P} \otimes M_{\cal L}) =0$.
Looking at the exact sequence
\[
0 \to M_{\cal L} \otimes  {\cal P}  \to H^0( {\cal L} ) \otimes  {\cal P}  \to  {\cal L}  \otimes  {\cal P}  \to 0
\]
and passing to global sections,
we have that $h^0( M_{\cal L} \otimes  {\cal P}  )=0$ if and only if the multiplication map 
$H^0( {\cal P} ) \otimes H^0( {\cal L} ) \to H^0(  {\cal P}  \otimes  {\cal L} )$
is injective.

Let us call $(\omega_C) - C_{g-3} + C$ the subset of $\picard^{g+2} (C)$ consisting of line bundles of the
form $\omega_C (-x_1 -x_2 -  ... - x_{g-3} + y)$ for some points
$x_1, \dots, x_{g-3}, y \in C$.
This is a 2-codimensional subset of $\picard^{g+2} (C)$, and 
its cohomology class is $W_{g-2} = \vartheta^2 /2$.

The elements
${\cal P} \in (\omega_C) - C_{g-3} + C$ are exactly those satisfying one of the following properties:
\begin{enumerate}
\item $h^0({\cal P}) > 3$
\item $h^0({\cal P}) = 3$ and ${\cal P}$ has base points.
\end{enumerate}
It can be shown that all elements satisfying one of these properties lie in 
$\Theta_{M_{\cal L}} $.
The remaining elements of $\picard^{g+2} (C)$ are line bundles ${\cal P} \in \picard^{g+2} (C)$
which are base point free and verifying $h^0({\cal P}) = 3$.

Let us show that there exists a line bundle ${\cal P} \in (\omega_C) - C_{g-3} + C$ such that the map 
$H^0({\cal P}) \otimes H^0({\cal L}) \to H^0( {\cal P} \otimes {\cal L})$
is injective.
%
Let us start by considering the multiplication map 
$$
\mu\colon H^0(D)\otimes H^0(\omega_C(-D))\longrightarrow H^0(\omega_C),
$$
and suppose that it is injective (then in fact it is an isomorphism): this assumption is true for a general curve, for instance it is true if we suppose that $C$ is a Petri curve.

Let $G$ be a general effective  divisor of degree $k+1$. 
Then observe that, as $G$ imposes general conditions on $H^0(\omega_C(-D))$,   
$$h^1(D+G)=h^0(\omega_C(-D-G))=0,$$
and hence 
$h^0(D+G)=3$. 

Let us prove that  $D+G$ is free from base points. Consider $p\in C$. If $p\in \mbox{supp} (G)$ then $h^0(\omega_C(-G+p))=0$, as the points of $G$ are in general position, and so by Riemann Roch $h^0(D+G-p)=2$. Let $p\not\in \mbox{supp} (G)$, and suppose by contradiction that $p$ is a base point of $D+G$. Then $p$ lies in the support of $D$, and  by  Riemann-Roch  we have that $h^0(\omega_C(-D-G+p))=1$. So  $G$ is special, contrary to the assumption. 

Hence $\oo_C(D+G)\in \picard^{g+2} (C)$ belongs to $ (\omega_C) - C_{g-3} + C$.

Let us now prove that the map
$\nu\colon H^0(D+G)\otimes H^0(\omega_C(-D)) \longrightarrow H^0(\omega_C(G))$
is injective.
Let $\sigma_1, \sigma_2, \sigma_3$ be a basis for $H^0(D+G)$ such that $\sigma_1$ and $\sigma_2$ generate $H^0(D)\subset H^0(D+G)$. 
Then of course the restriction of $\nu$ to $\langle \sigma_1,\sigma_2\rangle \otimes H^0(\omega_C(-D))$ is the map $\mu$, and so it is injective by our assumption.

Let $t=\ell_1\otimes\sigma_1+\ell_2\otimes \sigma_2+\ell_3\otimes\sigma_3$ be an element of $\ker \nu$, where  the $\ell_i$'s belong to $H^0(\omega_C(-D))$. 
Note that $G$ is the base locus of $\sigma_1$ and $\sigma_2$, and clearly $\ell_1\sigma_1+\ell_2\sigma_2\in H^0(\omega_C)\subset H^0(\omega_C(G))$.
As $-\ell_3\sigma_3=\ell_1\sigma_1+\ell_2\sigma_2$ in $H^0(\omega_C(G))$, and $\sigma_3$ does not vanish on any of the points of $G$, we have that $\ell_3$ has to vanish on $G$, but, as observed above, $H^0(\omega_C(-D-G))=\{0\}$. So $t=\ell_1\otimes\sigma_1+\ell_2\otimes \sigma_2$ but then $t$ is in $\ker \mu=\{0\}$, as wanted.

\end{proof}

\begin{remark}
Using the same notations as in the proof of the theorem above, let us make some remarks.

All line bundles ${\cal P} \in \picard^{g+2} (C) \setminus ((\omega_C) - C_{g-3} + C)$
induce a semistable dual span $M_{\cal P}$ of rank 2: in fact $M_{\cal P}$ is clearly a rank 2 bundle,
and any possible destabilization $Q \subset M_{\cal P}$ would be a line bundle of negative degree 
$-q > \mu (M_{\cal P}) = -(g +2)/2 = -( k+1 )= -\gamma (C)$,
and dualizing we would have a globally generated line bundle $Q^*$ of degree 
$q < \gamma (C)$, which is impossible.
By a similar argument it can be shown that $M_{\cal P}$ is actually stable for a general 
${\cal P} \in \picard^{g+2} (C)$.

Furthermore for all such ${\cal P}$, the bundle $M_{\cal P}$ admits a theta-divisor,
in fact all rank 2 stable bundles admit a theta-divisor
(very ample, cf. \cite{BV}).

Then we have a map
\[
\begin{array}{ccc}
\picard^{g+2} (C) \setminus ((\omega_C) - C_{g-3} + C) & \longrightarrow &
\mathcal{H}{ilb} (\picard^{3k}(C), 2 \vartheta)\\
{\cal P} & \mapsto & \Theta_{M_{\cal P}} \subset \picard^{3k} (C).
\end{array}
\]
%
%
%

\end{remark}

\begin{remark}
As it was shown in some cases that there exist DSB's $M_Q$ such that some exterior power
$\bigwedge^t M_Q$ has integral slope and does not admit a theta-divisor
(cf. \cite{popa}), it seems natural to ask the following question:

\end{remark}

\begin{question}

Do exterior powers $\bigwedge^t M_{\cal L}$ admit a theta-divisor?
If this is the case, are all theta-divisors integral?

\end{question}

In case both questions are answered affirmatively, it follows in particular that $M_{\cal L}$ is cohomologically stable in this case as well.


%

%

\begin{thebibliography}{ACGH85}

\bibitem[ACGH85]{ACGH}
E.~Arbarello, M.~Cornalba, P.~A. Griffiths, and J.~Harris, \emph{Geometry of
  algebraic curves. {V}ol. {I}}, Grundlehren der Mathematischen Wissenschaften
  [Fundamental Principles of Mathematical Sciences], vol. 267, Springer-Verlag,
  New York, 1985.

\bibitem[BBPN08]{BPN}
U.~N. Bhosle, L.~Brambila-Paz, and P.~E. Newstead, \emph{On coherent systems of
  type {$(n,d,n+1)$} on {P}etri curves}, Manuscripta Math. \textbf{126} (2008),
  no.~4, 409--441.

\bibitem[Bea95]{bove}
A.~Beauville, \emph{Vector bundles on curves and generalized theta
  functions: recent results and open problems}, Current topics in complex
  algebraic geometry ({B}erkeley, {CA}, 1992/93), Math. Sci. Res. Inst. Publ.,
  vol.~28, Cambridge Univ. Press, Cambridge, 1995, pp.~17--33.

\bibitem[BP08]{BP}
L.~Brambila-Paz, \emph{Non-emptiness of moduli spaces of coherent systems},
  Internat. J. Math. \textbf{19} (2008), no.~7, 779--799.

\bibitem[BPO09]{BPOrt}
L.~Brambila-Paz and A.~Ortega, \emph{Brill-{N}oether bundles and coherent
  systems on special curves}, Moduli spaces and vector bundles, London Math.
  Soc. Lecture Note Ser., vol. 359, Cambridge Univ. Press, Cambridge, 2009,
  pp.~456--472.

\bibitem[BS08]{barstop}
M.~A. Barja and L.~Stoppino, \emph{Linear stability of projected
  canonical curves with applications to the slope of fibred surfaces}, J. Math.
  Soc. Japan \textbf{60} (2008), no.~1, 171--192.

\bibitem[But94]{ButNormal}
D.~C. Butler, \emph{Normal generation of vector bundles over a curve}, J.
  Differential Geom. \textbf{39} (1994), no.~1, 1--34.

\bibitem[But97]{Butler}
D.~C. Butler, \emph{Birational maps of moduli of brill-noether pairs}, e-print
  \texttt{arxiv:alg-geom/9705009}, May 1997.

\bibitem[BV96]{BV}
S.~Brivio and A.~Verra, \emph{The theta divisor of {${\rm
  SU}_C(2,2d)^s$} is very ample if {$C$} is not hyperelliptic}, Duke Math. J.
  \textbf{82} (1996), no.~3, 503--552.

\bibitem[Cil83]{Ciliberto}
C.~Ciliberto, \emph{The degree of the generators of the canonical ring of a
  surface of general type}, Rend. Sem. Mat. Univ. Politec. Torino \textbf{41}
  (1983), no.~3, 83--111 (1984).

\bibitem[DN89]{DN}
J.-M. Drezet and M.~S. Narasimhan, \emph{Groupe de {P}icard des vari\'et\'es de
  modules de fibr\'es semi-stables sur les courbes alg\'ebriques}, Invent.
  Math. \textbf{97} (1989), no.~1, 53--94.

\bibitem[EL92]{EL}
L.~Ein and R.~Lazarsfeld, \emph{Stability and restrictions of
  {P}icard bundles, with an application to the normal bundles of elliptic
  curves}, Complex projective geometry ({T}rieste, 1989/{B}ergen, 1989), London
  Math. Soc. Lecture Note Ser., vol. 179, Cambridge Univ. Press, Cambridge,
  1992, pp.~149--156.

\bibitem[Gre84]{green}
M.~L. Green, \emph{Koszul cohomology and the geometry of projective
  varieties}, J. Differential Geom. \textbf{19} (1984), no.~1, 125--171.

\bibitem[HL97]{Huy}
D.~Huybrechts and M.~Lehn, \emph{The geometry of moduli spaces of
  sheaves}, Aspects of Mathematics, E31, Friedr. Vieweg \& Sohn, Braunschweig,
  1997.

\bibitem[Mis08]{stabtrans}
E.~C. Mistretta, \emph{Stability of line bundle transforms on curves with
  respect to low codimensional subspaces}, J. Lond. Math. Soc. (2) \textbf{78}
  (2008), no.~1, 172--182.

\bibitem[Mum77]{Mum}
D.~Mumford, \emph{Stability of projective varieties}, Enseignement Math. (2)
  \textbf{23} (1977), no.~1-2, 39--110.

\bibitem[Pop97]{popasurvey}
M.~Popa, \emph{Generalized theta linear series on moduli spaces of vector
  bundles on curves}, e-print \texttt{arXiv:0712.3192}, December 1997.

\bibitem[Pop99]{popa}
M.~Popa, \emph{On the base locus of the generalized theta divisor}, C. R. Acad.
  Sci. Paris S\'er. I Math. \textbf{329} (1999), no.~6, 507--512.

\bibitem[PR88]{PR}
K.~Paranjape and S.~Ramanan, \emph{On the canonical ring of a curve},
  Algebraic geometry and commutative algebra, {V}ol.\ {II}, Kinokuniya, Tokyo,
  1988, pp.~503--516.

\bibitem[Sto08]{LS}
L.~Stoppino, \emph{Slope inequalities for fibred surfaces via {GIT}}, Osaka
  J. Math. \textbf{45} (2008), no.~4, 1027--1041.

\end{thebibliography}

%

\vspace{1cm}


\noindent
 \begin{minipage}[t]{5cm}

\begin{flushleft}
\small{

Ernesto C. \textsc{Mistretta}

\texttt{ernesto@math.unipd.it}

\textsc{Università di Padova}

Dipartimento di Matematica

Via Trieste 63 

 35121 Padova - Italy}

\end{flushleft}

\end{minipage}
\hfill
\begin{minipage}[t]{5cm}

\begin{flushright}

\small{

Lidia \textsc{Stoppino}

\texttt{lidia.stoppino@uninsubria.it}

\textsc{Università dell'Insubria}

Dipartimento di Matematica

Via Valleggio 11 

 22100 Como - Italy}

\end{flushright}

\end{minipage}

\end{document}